\documentclass[11pt]{amsart}

\usepackage[english]{babel}
\usepackage[latin1]{inputenc}
\usepackage[T1]{fontenc}
\usepackage{lmodern}
\usepackage{amssymb}
\usepackage{amsmath}
\usepackage{enumerate}

\newtheorem{theorem}{Theorem}[section]
\newtheorem{lemma}[theorem]{Lemma}
\newtheorem{proposition}[theorem]{Proposition}
\newtheorem{corollary}[theorem]{Corollary}
\newtheorem{remark}[theorem]{Remark}

\def\og{\leavevmode\raise.3ex\hbox{$\scriptscriptstyle\langle\!\langle$~}}
\def\fg{\leavevmode\raise.3ex\hbox{~$\!\scriptscriptstyle\,\rangle\!\rangle$}}

\newcommand{\N}{\mathbb{Z}_{\geq 0}}
\newcommand{\Z}{\mathbb{Z}}
\newcommand{\R}{\mathbb{R}}
\newcommand{\C}{\mathbb{C}}
\newcommand{\Sp}{\mathbb{S}}

\def\germ #1 {\mathfrak{#1}}
\def\cal #1 {\mathcal{#1}}

\title{Integrability of weight module of degree 1}

\author{Guillaume Tomasini}

\begin{document}


\maketitle


\begin{abstract}
The aim of this article is to find all weight modules of degree $1$ of a simple complex Lie algebra that integrate to a continuous representation of a simply-connected real Lie group on some Hilbert space.
\end{abstract}


\section{Introduction}

Let $\germ g $ denote a simple Lie algebra over the field $\C$ of complex numbers. Let $\germ h $ be a Cartan subalgebra of $\germ g $. A weight module is a $\germ g $-module, $\germ h $-diagonalizable, having finite dimensional weight spaces. The set of all weight $\germ g $-modules is a category containing the BGG category $\cal O $. This category of weight modules has been much studied in recent years (e.g. \cite{Fe90,Fu,BL87,Ma00,BKLM04,BL01,GS06,GS10,MS10,MS10b}). It is then a natural question to find those weight modules that integrate to continuous (resp. unitary) representations of simply-connected real Lie groups. They should form a small but interesting class of representations, with small Gelfand-Kirillov dimension. In this paper, we treat the case of weight modules with weight multiplicities equal to $1$.

Let us explain our strategy. Let $G$ be a simply-connected real Lie group. Assume the $\germ g $-module $V$ integrates to a continuous representation $\pi$ of $G$ in some Hilbert space $\cal H $. Let $K$ be a compact subgroup of $G$ and denote by $(\pi_{|K},\cal H _{|K})$ the representation $(\pi,\cal H )$ restricted to $K$. Then it is well known that the representation $(\pi_{|K},\cal H _{|K})$ is unitarizable. Therefore, we can express $\cal H _{|K}$ as a direct sum of simple finite dimensional unitary representation of $K$. A consequence is that the (complexified) Lie algebra of $K$ should act nicely on $V$: the module $V$ should be a $\germ k $-finite $\germ g $-module. If $K$ is big enough, this condition is strong enough to imply that $V$ should be a highest (or lowest) weight module. We then use a classification result due to Benkart, Britten and Lemire to describe the possible modules. Then it remains to check whether or not these modules can be integrated. Our results also make use of a theorem of J{\o}rgensen and Moore and classical results about discrete series.


\section{Some facts about weight modules}

\subsection{Weight modules of degree $1$}

Let $\germ g $ denote a simple Lie algebra over the field $\C$ of complex numbers. Fix a Cartan subalgebra $\germ h $. A $\germ g $-module $V$ is called a \emph{weight module} if 
\begin{enumerate}
\item The module $V$ is finitely generated,
\item We have the following decomposition of $V$: $$V=\bigoplus_{\lambda \in \germ h ^*}\ V_{\lambda},\quad V_{\lambda}:=\{m\in V\ | \ \forall \ H \in \germ h , \ H\cdot m =\lambda(H)m\},$$
\item The weight spaces $V_{\lambda}$ are all finite dimensional.
\end{enumerate}
We call \emph{degree} of a weight module the supremum of the dimension of the weight spaces:
$$deg(V)=\sup_{\lambda}\: \{dim(V_{\lambda})\} \in \N\cup \infty.$$ When $deg(V)\in \N$, we call $V$ a \emph{bounded module}. In particular, if $V$ is a weight $\germ g $-module of degree 1, then all the non zero weight spaces are $1$-dimensional. The weight modules of  degree 1 have been studied by Benkart, Britten and Lemire \cite{BBL97}. In particular, they constructed weight modules of degree 1, $N(a)$ (with $a\in \C^n$) for $\germ sl (n+1,\C)$ and $M(b)$ (with $b\in \C^n$) for $\germ sp (n,\C)$ (see \cite{BBL97} for the explicit construction of these modules). Moreover, they proved the following:
\begin{theorem}[Benkart, Britten, Lemire {\cite{BBL97}}]\label{thm-BBL}
Let $V$ be a simple infinite dimensional weight $\germ g $-module of degree 1. Then
\begin{enumerate}
\item The Lie algebra $\germ g $ is isomorphic to either $\germ sl (n+1,\C)$ or to $\germ sp (n,\C)$.
\item The Gelfand-Kirillov dimension of $V$ equals the rank of $\germ g $.
\item If $\germ g =\germ sl (n+1,\C)$, then there is $a\in \C^n$ such that $V\cong N(a)$.
\item If $\germ g =\germ sp (n,\C)$, then there is $b\in \C^n$ such that $V\cong M(b)$.
\end{enumerate}
\end{theorem}

For our purpose we shall need another notion. Let $\germ l $ be a subalgebra of $\germ g $. A $\germ g $-module $V$ is a \emph{$(\germ g ,\germ l )$-module of finite type} if as an $\germ l $-module, $V$ splits into a direct sum of simple finite dimensional $\germ l $-modules, with finite multiplicities. For instance, a weight module is a $(\germ g ,\germ h )$-module of finite type. The general notion of $(\germ g ,\germ k )$-module has been studied in details by Penkov, Serganova and Zuckerman in \cite{PS02,PSZ04,PZ04a,PZ04b,PZ07}.

\subsection{Classification of $(\germ g ,\germ l _{j})$-module of finite type and of degree $1$}

Let $n$ be a positive integer greater than $1$. Let $\germ g $ denote the Lie algebra $\germ sl (n+1,\C)$. Let $\germ h $ denote the standard Cartan subalgebra of $\germ g $, consisting of diagonal matrices. Denote by $H_0, H_1, \ldots, H_{n-1}$ its canonical basis. Denote by $E_j$ (resp. $F_j$) the vector in $\germ g $ corresponding to the elementary matrix $E_{j+1,j+2}$ (resp. $E_{j+2,j+1}$) for $0\leq j\leq n-1$. Then $\germ g $ is Lie-generated by the vectors $\{H_j,E_j,F_j\}_{0\leq j\leq n-1}$. Denote by $\germ l _j$ the maximal standard Levi subalgebra of $\germ g $ Lie-generated by $\germ h $ and the vectors $\{E_k,F_k\}_{k\not=j}$.

For future use, we shall find those infinite dimensional weight $\germ g $-modules of degree 1 whose restriction to some $\germ l _j$ is a direct sum of finite dimensional $\germ l _j$-modules. To this aim, we need the following general result:
\begin{lemma}[Fernando {\cite{Fe90}}, Benkart-Britten-Lemire {\cite{BBL97}}]\label{Fe-BBL}
Let $\germ a $ be a simple Lie algebra over $\C$. Let $\germ t $ be a Cartan subalgebra of $\germ a $. Let $V$ be a simple weight $\germ a $-module. Let $\cal R $ denote the root system of $(\germ a ,\germ t )$. Then
\begin{enumerate}
\item For any $\alpha\in \cal R $, and any $X\in\germ a _{\alpha}\setminus\{0\}$, the action of $X$ on $V$ is either locally finite or injective.
\item Let $\alpha, \beta \in \cal R $ be such that there are $X^{\pm}\in\germ a _{\pm\alpha}\setminus\{0\}$ and $Y^{\pm}\in \germ a _{\pm\beta}\setminus\{0\}$ satisfying $X^{\pm}$ both act locally finitely on $V$ and $Y^{\pm}$ both act injectively on $V$. Then $\alpha+\beta\not\in \cal R $.
\end{enumerate}
\end{lemma}
\begin{proof} See \cite[Section4]{BBL97}.
\end{proof}

\begin{corollary}\label{cor-gKd1}
Let $V$ be a simple weight $\germ g $-module. Let $0\leq j\leq n-1$. Assume that $V$ is a $(\germ g ,\germ l _j)$-module of finite type. Then $V$ is a highest weight or a lowest weight module.
\end{corollary}
\begin{proof} From lemma \ref{Fe-BBL}, the vectors $E_k$ and $F_k$ act locally finitely or injectively on $V$. By hypothesis, they act locally finitely for $k\not=j$. Now, lemma \ref{Fe-BBL} applied to $E_j$, $F_j$ and either $E_{j-1}$, $F_{j-1}$ or $E_{j+1}$, $F_{j+1}$ shows that at least one of the vectors $E_j$ and $F_j$ acts locally finitely on $V$. In the first case, the module is a highest weight module. In the second case, it is a lowest weight module.
\end{proof}

Denote by $\{\omega_{i}\}_{i=1..n}$ the fundamental weights for $\germ g $. Recall now the following:

\begin{proposition}[Benkart-Britten-Lemire {\cite[Proposition 3.4]{BBL97}}]\label{prop-HWd1}
Up to isomorphism, the only highest weight $\germ g $-module of degree 1 are the modules with highest weight $a\omega_{1}$, $a\omega_{i}-(1+a)\omega_{i+1}$, and $a\omega_{n}$, for $a\in \C$.
\end{proposition}

In the notation of theorem \ref{thm-BBL}, these modules correspond to $N(a,0,\ldots,0)$, $N(\underbrace{-1,\ldots,-1}_{i},-1-a,0,\ldots ,0)$, and $N(-1,\ldots,-1,-1-a)$, for $a\in \C$.

In the sequel we will often work with the $\germ sl (n+1,\C)$-module $N(a,0,\ldots,0)$. For the convenience of the reader we write down here the action of the vectors $\{H_j,E_j,F_j\}_{0\leq j\leq n-1}$ on a basis. The module $N(a,0,\ldots,0)$ has a basis $x(k)$ indexed by $k\in \N^n$. If $k=(k_1,\ldots,k_n)\in \N^n$, we set $|k|:=k_1+\cdots +k_n$. The action is given by:
\begin{subequations}
\label{eq-Na}
\begin{align}
H_0\cdot x(k) & = (a-k_1-|k|)x(k)\\
H_j\cdot x(k) & = (k_j-k_{j+1})x(k)\\
E_0\cdot x(k) & = k_1x(k-\epsilon_1)\\
F_0\cdot x(k) & = (a-|k|)x(k+\epsilon_1)\\
E_j\cdot x(k) & = k_{j+1}x(k-\epsilon_{j+1}+\epsilon_j)\\
F_j\cdot x(k) & = k_jx(k+\epsilon_{j+1}-\epsilon_j)
\end{align}
\end{subequations}

From this classification, we are in position to prove the

\begin{proposition}\label{prop-glmod}
Let $0\leq j\leq n-1$. Let $V$ be a simple infinite dimensional $(\germ g ,\germ l _j)$-module of finite type and of degree 1. Then 
\begin{enumerate}
\item If $j=0$, then $V$ or its contragredient is isomorphic to $N(a,0,\ldots,0)$, for some $a\in \C\setminus \N$ or to $N(-1,m,0,\ldots,0)$ for some $m\in \N$.
\item If $j=n-1$, then $V$ or its contragredient is isomorphic to $N(-1,\ldots,-1,a)$, for some $a\in \C\setminus\Z_{<0}$ or to $N(-1,\ldots,-1,-1-m,0)$ for some $m\in \N$.
\item If $0<j<n-1$, then $V$ or its contragredient is isomorphic to\newline $N(\underbrace{-1,\ldots,-1}_{j+1},m,0,\ldots,0)$ or $N(\underbrace{-1,\ldots,-1}_{j},-1-m,0,\ldots,0)$, for some $m\in \N$.
\end{enumerate}
\end{proposition}
\begin{proof} From corollary \ref{cor-gKd1}, $V$ or its contragredient is a simple highest weight module. Therefore we know that $V$ or its contragredient is given by proposition \ref{prop-HWd1}. It thus remains to check whether or not the modules in proposition\ref{prop-HWd1} satisfy the restriction property.

Assume that $j=0$. Let $0<k<n-1$ and consider the module $V=N(-1,\ldots,-1,a,0,\ldots,0)$ where $a$ is a complex number in position $k+1$. Then the highest weight $x$ for this module satisfies:
$$H_{k-1}\cdot x=(-1-a)x,\ H_k\cdot x=ax.$$ If $k-1\not=0$, then by our assumption on $V$, $x$ should generate a finite dimensional module. This imposes that the vectors $H_1,\ldots,H_{n-1}$ of the Cartan subalgebra acts on $x$ by non negative integers. Therefore $a$ and $-1-a$ should be non negative integers, which is impossible. The same argument shows that for $k=1$ we must have $a\in \N$.

Consider now the module $N(-1,\ldots,-1,a)$ for $a\in \C$. Using once again the same argument, we show that necessarily $a$ should be a negative integer. In this case, the module is finite dimensional.

We need to show now that $N(a,0,\ldots,0)$ for $a\in \C\setminus\N$ and\newline $N(-1,m,0,\ldots,0)$ for $m\in \N$ do satisfy the restriction property. Let us prove it for $N(a,0\ldots,0)$.

We want to find those linear combinations $\sum\ \mu_kx(k)$ which are highest weight vectors for the action of $\germ l _0$. From the explicit action given by formulae \eqref{eq-Na}, we conclude that the highest weight vectors are the linear combinations of the following linearly independent highest weight vectors:
$$x(k_1,0,\ldots,0),\ k_1\in \N.$$ We shall prove now that the module $\cal U (\germ l _0)x(k_1,0,\ldots,0)$ is a simple highest weight module. Since it is a highest weight module, it is indecomposable. It is simple if and only if it does not contain a highest weight vector linearly independent of $x(k_1,0,\ldots,0)$. But any such vector is a linear combination of $x(k',0,\ldots,0)$ for some $k'\in \N$. However the action of $nH_0+(n-1)H_1+\cdots +H_{n-1}$, vector generating the center of $\germ l _0$, on $x(k',0,\ldots,0)$ is:
$$(nH_0+(n-1)H_1+\cdots +H_{n-1}) \cdot x(k',0,\ldots,0)=na-(n+1)k'.$$ Therefore the center acting as a scalar on $\cal U (\germ l _0)x(k_1,0,\ldots,0)$, we conclude that there is no highest weight vector in this module but the multiples of $x(k_1,0,\ldots,0)$. Thus proving that the module is simple. Hence it is clear that we have the following branching:
$$N(a,0,\ldots,0)_{|\germ l _0}=\bigoplus_{k\in \N}\ \cal U (\germ l _0)x(k,0,\ldots,0).$$ Therefore we proved that $N(a,0,\ldots,0)$ is a $(\germ g ,\germ l _0)$-module of finite type as asserted. The proof for $N(-1,m,0,\ldots,0)$ is the same.

The case $j>0$ is analogous.
\end{proof}


\section{Type $A$ case}

In this section, we shall find which degree $1$ $\germ sl (n+1,\C)$-module integrate to a continuous representation of some real Lie group whose complexified Lie algebra is $\germ sl (n+1,\C)$.

\subsection{A natural action of $SU(1,n)$}

Let $n$ be a positive integer. Let $SU(1,n)$ denote the subgroup of $GL(n+1,\C)$ consisting of those matrices $g$ such that 
$$^t\!\bar{g}\times \left(\begin{array}{cc}
-1 & \\
 & I_{n}\end{array}\right) \times g = \left(\begin{array}{cc}
-1 & \\
 & I_{n}\end{array}\right)$$ and whose determinant is $1$. We shall label rows and columns from $0$ to $n$. This is a real Lie group. It acts on $\Sp_{n}:=\{(z_{j})\in \C^n\ | \ \sum_{j=1}^n\ |z_{j}|^2=1\}$ via 
 $$g\cdot \left(\begin{array}{c}
 z_{1}\\
 \vdots\\
 z_{n}\end{array}\right) = \left(\begin{array}{c}
 \frac{\displaystyle\sum_{j=0}^n\ g^j_{1}z_{j}}{\displaystyle\sum_{j=0}^n\ g^j_{0}z_{j}}\\
\vdots\\
\frac{\displaystyle\sum_{j=0}^n\ g^j_{n}z_{j}}{\displaystyle\sum_{j=0}^n\ g^j_{0}z_{j}}\end{array}\right),$$ where $g=(g^j_{k})_{0\leq j,k\leq n}\in SU(1,n)$ and $z_{0}=1$. Since $SU(1,n)$ preserves the quadratic form $-|Z_{0}|^2+|Z_{1}|^2+\cdots +|Z_{n}|^2$, the denominator is never $0$ and $g\cdot z$ is in $\Sp_{n}$ for any $z\in \Sp_{n}$. Denote by $d\sigma$ the measure on $\Sp_{n}$ induced from the Lebesgue measure of $\C^n$ and by $\Omega_{n}$ the volume of $\Sp_{n}$. It is well known that $\Omega_{n}=\frac{2\pi^n}{(n-1)!}$. We also denote by $\cal H (\C^n)$ the space of holomorphic functions from $\C^n$ to $\C$. Then the action of $SU(1,n)$ on $\Sp_{n}$ induces a natural continuous representation on $L^2(\Sp_{n},\frac{d\sigma}{\Omega_{n}})\cap \cal H (\C^n)$. We can further construct a unitary representation $\rho$ on this space by 
$$\rho(g)(\varphi)(z):=\left(\sum_{j=0}^n\ (g^{-1})^j_{0}z_{j}\right)^{-n}\times \varphi(g^{-1}\cdot z).$$

Let $k=(k_{j})\in \N^n$. Set $P(k)(z):=\displaystyle\prod_{j=1}^n\ z_{j}^{k_{j}}$. Then the family $(P(k))_{k\in \N^n}$ is an orthogonal basis for the Hilbert space $L^2(\Sp_{n},\frac{d\sigma}{\Omega_{n}})\cap \cal H (\C^n)$. Moreover, we have $\|P(k)\|^2=\frac{\displaystyle\prod_{j=1}^n\ k_{j}!}{\displaystyle\prod_{j=1}^{|k|}\ (j+n-1)}$, where $|k|:=\displaystyle\sum_{j=1}^n\ k_{j}$.

Consider the following $1$-parameter families:
$$e^{itH_{1}}:=\left(\begin{array}{cccccc}
1 & & & & & \\
 & e^{-it} & & & & \\
 & & e^{it} & & & \\
 & & & 1 & & \\
 & & & & \ddots & \\
 & & & & & 1\end{array}\right), \ldots , e^{itH_{n-1}}:=\left(\begin{array}{ccccc}
1 & & & & \\
 & \ddots & & & \\
 & & 1 & & \\
 & & & e^{-it} & \\
 & & & & e^{it}\end{array}\right).$$
 $$e^{tX_{1}}:=\left(\begin{array}{cccccc}
1 & & & & & \\
 & \cos t & -\sin t & & & \\
 & \sin t & \cos t & & & \\
 & & & 1 & & \\
 & & & & \ddots & \\
 & & & & & 1\end{array}\right),e^{tY_{1}}:=\left(\begin{array}{cccccc}
1 & & & & & \\
 & \cos t & -i\sin t & & & \\
 & -i\sin t & \cos t & & & \\
 & & & 1 & & \\
 & & & & \ddots & \\
 & & & & & 1\end{array}\right), \ldots, $$
 $$e^{tX_{n-1}}:=\left(\begin{array}{ccccc}
1 & & & & \\
 & \ddots & & & \\
 & & 1 & & \\
 & & & \cos t & -\sin t \\
 & & & \sin t & \cos t \end{array}\right), e^{tY_{n-1}}:=\left(\begin{array}{ccccc}
1 & & & & \\
 & \ddots & & & \\
 & & 1 & & \\
 & & & \cos t & -i\sin t \\
 & & & -i\sin t & \cos t \end{array}\right),$$
 $$e^{tX_{0}}:=\left(\begin{array}{ccccc}
\cosh t & -\sinh t & & & \\
-\sinh t & \cosh t & & & \\
 & & 1 & & \\
 & & & \ddots & \\
 & & & & 1\end{array}\right), e^{tY_{0}}:=\left(\begin{array}{ccccc}
\cosh t & -i\sinh t & & & \\
i\sinh t & \cosh t & & & \\
 & & 1 & & \\
 & & & \ddots & \\
 & & & & 1\end{array}\right), $$
 $$e^{itH_{0}}:=\left(\begin{array}{ccccc}
e^{-it} & & & & \\
 & e^{it} & & & \\
 & & 1 & & \\
 & & & \ddots & \\
 & & & & 1\end{array}\right).$$
 Then the Lie algebra $\germ su (1,n)$ of $SU(1,n)$ is generated (as a Lie algebra) by $$iH_{0},\ldots ,iH_{n-1},X_{0},\ldots ,X_{n-1},Y_{0},\ldots Y_{n-1}.$$ Set 
 $$E_{0}:=\frac{X_{0}+iY_{0}}{2},\ F_{0}:=\frac{X_{0}-iY_{0}}{2},$$
 $$E_{j}:=-\frac{X_{j}+iY_{j}}{2},\ F_{j}:=\frac{X_{j}-iY_{j}}{2},\ 1\leq j\leq n-1.$$ Then $(H_{j},E_{j},F_{j})_{0\leq j\leq n}$ generates a Lie algebra $\germ g $ isomorphic to $\germ sl (n+1,\C)$. The Cartan subalgebra $\germ h $ of $\germ g $ is the subalgebra generated by $\{H_{0},\ldots, H_{n-1}\}$. We can compute as usual the action of $\germ g $ on the basis $(P(k))_{k\in \N^n}$. We get:
\begin{subequations}
\label{eq-action}
\begin{align}
 & \left\{\begin{array}{ccc}
 H_{0}\cdot P(k) & = & (-n-|k|-k_{1})P(k)\\
 E_{0}\cdot P(k) & = & k_{1}P(k-\epsilon_{1})\\
 F_{0}\cdot P(k) & = & (-n-|k|)P(k+\epsilon_{1})\end{array}\right.,\\
 & \left\{\begin{array}{ccc}
 H_{j}\cdot P(k) & = & (k_{j}-k_{j+1})P(k)\\
 E_{j}\cdot P(k) & = & k_{j+1}P(k-\epsilon_{j+1}+\epsilon_{j})\\
 F_{j}\cdot P(k) & = & k_{j}P(k+\epsilon_{j+1}-\epsilon_{j})\end{array}\right.,\ \forall\ 1\leq j\leq n-1,\end{align}
 \end{subequations} where $\epsilon_{j}$ is the vector in $\N^n$ whose entries are all zero except the $j$th entry which is $1$.
 
In the sequel we shall deform this infinitesimal representation and show the $1$-parameter deformation thus constructed integrates to a continuous representation of the universal cover of $SU(1,n)$. We shall further explicit those values of the parameter such that the representation is unitary.
 
 \subsection{Deformation of the natural action of $\germ sl (n+1,\C)$}
 
To each $k\in \N^n$, let us associate a vector $e(k)$. Let 
$$\cal H :=\left\{u:=\sum_{k\in \N^n}\ u_{k}e(k)\ \left| \ \sum_{k\in \N^n}\ |u_{k}|^2\frac{\displaystyle\prod_{j=1}^n\ k_{j}!}{\displaystyle\prod_{j=1}^{|k|}\ (j+n-1)}<\infty\right. \right\}.$$ We define on $\cal H $ a Hilbert space structure by requiring that the basis $(e(k))$ is orthogonal and that $\|e(k)\|^2=\frac{\displaystyle\prod_{j=1}^n\ k_{j}!}{\displaystyle\prod_{j=1}^{|k|}\ (j+n-1)}$. Denote by $G$ the universal cover of $SU(1,n)$. According to the previous subsection, there is a continuous representation $\rho$ of $G$ (in fact, a unitary representation of $SU(1,n)$) corresponding to the representation of $\germ g $ given on $\cal H $ by formulae \eqref{eq-action}.

Let $a\in \C\setminus\N$. For any $l\in \N$, set $\mu_{a}(l):=\sqrt{\displaystyle\prod_{j=1}^l\ \frac{j+n-1}{|j-a-1|}}$. Note that $\mu_{a}(l)$ is a well defined positive real number such that $\mu_{-n}(l)=1$. We define now operators $(H_{j}(a),E_{j}(a), F_{j}(a))_{0\leq j\leq n-1}$ on $\cal H $ by their action on the basis $e(k)$ as follows:
$$\forall \ 1\leq j \leq n-1, H_{j}(a)=H_{j},\ E_{j}(a)=E_{j},\ F_{j}(a)=F_{j},$$
\begin{subequations}
\begin{align}
(H_{0}(a)-H_{0})\cdot e(k) & = (n+a)e(k),\\
(E_{0}(a)-E_{0})\cdot e(k) & = k_{1}\left(\frac{\mu(|k|-1)}{\mu(|k|)}-1\right)e(k-\epsilon_{1}),\\
(F_{0}(a)-F_{0})\cdot e(k) & = \left((a-|k|)\frac{\mu(|k|+1)}{\mu(|k|)}+n+|k|\right)e(k+\epsilon_{1}).
\end{align}
\end{subequations}
Notice that if $a=-n$ then all the new operators coincide with the corresponding undeformed operator. It is easy to check that the new operators give rise to another representation of $\germ sl (n+1,\C)$. We shall refer to the space of this representation as $\cal H _{a}$ (even though as a Hilbert space it is nothing but $\cal H $). In the sequel we will use 
$$m^-_{k}:=k_{1}\left(\frac{\mu(|k|-1)}{\mu(|k|)}-1\right),\ m^+_{k}:=(a-|k|)\frac{\mu(|k|+1)}{\mu(|k|)}+n+|k|.$$

\begin{remark}\label{rema-fdd1}
Set $x(k):=\mu(|k|)e(k)$. Then $\|x(k)\|^2=\frac{\displaystyle\prod_{j=1}^n\ k_{j}!}{\displaystyle\prod_{j=1}^{|k|}\ |j-a-1|}$. Moreover, $x(k)$ is also a basis for $\cal H _{a}$ and the deformed action of $\germ sl (n+1,\C)$ on $\cal H _{a}$ is precisely the one given in \cite{BBL97}.
\end{remark}

\subsection{Integrability of the representation $\cal H _{a}$}

To prove the integrability of the representation $\cal H _{a}$ we shall use a criterion of J{\o}rgensen and Moore \cite{JM84}. Let us recall it. Let $\cal H '$ be a Hilbert space. Let $ D $ be a dense subspace. Denote by $\|\cdot\|_{_{0}}$ the Hilbert norm. By $\cal A (D)$ we mean the set of all operators on $D$. Let $A_{0}=Id$ and let $A_{1},\ldots ,A_{d}$ be a basis for some (finite dimensional) Lie algebra included in $\cal A (D)$. We define inductively a norm $\|\cdot\|_{_{l}}$ on $D$ by setting $\|u\|_{_{l+1}}:=\max \{\|A_{k}u\|_{_{l}},\ 0\leq k\leq d\}$. Denote by $D_{l}$ the completion of $D$ with respect to $\|\cdot \|_{_{l}}$ and by $L_{j}$ the space of continuous operators of $D_{j}$. Let $L^{u}(D_{\infty}):=\cap\{L_{j},\ j\geq 0\}$ and $\cal A _{\infty}(D):=\cal A (D)\cap L^{u}(D_{\infty})$. This is the set of operators on $D$ bounded for all the norms $\|\cdot\|_{_{l}}$.

\begin{theorem}[J{\o}rgensen-Moore]
Let $G$ be a connected simply-connected Lie group, whose corresponding Lie algebra is denoted $\germ g _{\R}$. Let $\pi_{0}$ be a continuous representation of $G$ on $\cal H '$. Set $\cal L _{0}:=d\pi_{0}(\germ g _{\R})$. Let $D:=C^{\infty}(\pi_{0})$. Let $S_{0}$ be a set of Lie generator for $\cal L _{0}$. Let $f:S_{0}\rightarrow \cal A _{\infty}(D)$ be such that $S:=\{A+f(A),\ A\in S_{0}\}$ generates a finite dimensional Lie algebra $\cal L $. Then the representation $\cal L $ can be integrated into a continuous representation $\pi$ of $G$ such that $d\pi(\germ g _{\R})=\cal L $.
\end{theorem}

To apply the theorem to our situation, we set $S_{0}:=\{iH_{j}, X_{j}, Y_{j},\ 0\leq j\leq n-1\}$. The Hilbert space is $\cal H $, the dense subset is 
$$D:=\left\{u=\sum_{k\in \N^n}\ u_{k}e(k)\in \cal H \ \left| \ \sum_{k\in \N^n}\ |k|^Nu_{k}e(k)\in \cal H , \ \forall \ N\in \N \right. \right\}.$$ The function $f$ is given by the formulae \eqref{eq-action}. At this point, we need to check that the image of $f$ is in $\cal A _{\infty}(D)$, i.e. to check that the operators defined on $D$ by formulae \eqref{eq-action} are bounded for all the norms $\|\cdot\|_{_{l}}$. As $f(iH_{j})=f(X_{j})=f(Y_{j})=0$ for $j\geq 1$, we only need to consider the three operators $f(iH_{0})$, $f(X_{0})$ and $f(Y_{0})$. As $f(iH_{0})$ is a scalar operator, the boundedness is clear. We have to prove it for $f(X_{0})$ and $f(Y_{0})$. First we note the following:

\begin{lemma}
The operators $f(X_{0})$ and $f(Y_{0})$ are bounded for all the norms $\|\cdot\|_{_{l}}$ if and only if the operators $f(E_{0})$ and $f(F_{0})$ are.
\end{lemma}
\begin{proof} This is clear since $X_{0}$ and $Y_{0}$ are linear combination of $E_{0}$ and $F_{0}$ and vice-versa.
\end{proof}

\begin{lemma}\label{lem-l1}
The operators $f(E_{0})$ and $f(F_{0})$ are bounded for all the norms $\|\cdot\|_{_{l}}$ if and only if they are bounded for all the norms constructed using the set $\tilde{S}_{0}:=\{H_{j},E_{j},F_{j},\ 0\leq j\leq n-1\}$ instead of $S_{0}$.
\end{lemma}
\begin{proof} Again, this follows from the fact that we can express the elements of $S_{0}$ as linear combinations of those in $\tilde{S}_{0}$ and vice-versa.
\end{proof}

\begin{proposition}
The operators $f(E_{0})$ and $f(F_{0})$ are bounded for all the norms $\|\cdot\|_{_{l}}$.
\end{proposition}
\begin{proof} From the lemma \ref{lem-l1}, we can use the norms $\|\cdot\|_{_{l}}$ constructing from the set $\tilde{S}_{0}$. We prove the lemma by induction on $l$ for $f(E_{0})$. The proof for $f(F_{0})$ is analogous. Let $u:=\sum\ u_{k}e(k)\in D$. Then 
\begin{align*}
\|f(E_{0})u\|^2 & =  \sum_{k}\ |u_{k}|^2|m_{k}^-|^2\|e(k-\epsilon_{1})\|^2\\
 & = \sum_{k| k_{1}>0}\ |u_{k}|^2|m_{k}^-|^2\|e(k)\|^2\frac{|k|+n-2}{k_{1}}\\
 & = \sum_{k| k_{1}>0}\ |u_{k}|^2\|e(k)\|^2\times k_{1}(|k|+n-2)\left(\frac{\mu(|k|-1)}{\mu(|k|)}-1\right)^2\\
 & \leq \sup_{k| k_{1}>0} \left\{k_{1}(|k|+n-2)\left(\frac{\mu(|k|-1)}{\mu(|k|)}-1\right)^2\right\}\times \sum_{k}\ |u_{k}|^2\|e(k)\|^2\\
 & = \sup_{k| k_{1}>0} \left\{k_{1}(|k|+n-2)\left(\frac{\mu(|k|-1)}{\mu(|k|)}-1\right)^2\right\}\times \|u\|^2.
\end{align*}
Using an asymptotic development of $\left(\frac{\mu(|k|-1)}{\mu(|k|)}-1\right)^2$ we easily see that the above supremum is finite, thus proving that $f(E_{0})$ is bounded for $\|\cdot\|_{_{0}}$. Assume now that $f(E_{0})$ is bounded for the norms $\|\cdot\|_{_{l}}$ for $0\leq l\leq M-1$.

Let $A_{1},\ldots , A_{M}$ be elements in $\tilde{S}_{0}$. Let $u$ ba as above and consider the expression $\|A_{1}\cdots A_{M}f(E_{0})u\|^2$. Since $A_{1}\cdots A_{M}$ is a weight vector in the enveloping algebra, the vectors $A_{1}\cdots A_{M}e(k)$ are mutually orthogonal. Therefore, we have 
\begin{align}
\|A_{1}\cdots A_{M}f(E_{0})u\|^2 & = \sum_{k}\ |u_{k}|^2|m^-_{k}|^2\|A_{1}\cdots A_{M}e(k-\epsilon_{1})\|^2.
\end{align}
Now from formulae \eqref{eq-action}, it is clear that 
$$A_{1}\cdots A_{M}e(k)=P_{A_{1}\cdots A_{M}}(k)e(k+l(A_{1}\cdots A_{M})),$$ where $P_{A_{1}\cdots A_{M}}(k)$ is a polynomial in $(k_{1},\ldots,k_{n})$, product of $M$ monomials of degree $1$. For brevity, we shall denote it by $P(k)$ in the sequel. Moreover $l(A_{1}\cdots A_{M})\in \Z^n$ and $|l(A_{1}\cdots A_{M})|:=\displaystyle\sum_{j=1}^n\ |l(A_{1}\cdots A_{M})_{j}|$ is a non negative integer smaller than $2M$, since for any $A\in \tilde{S}_{0}$, the vector $Ae(k)$ is either a multiple of $e(k)$ or a multiple of $e(k\pm \epsilon_{1})$ or a multiple of $e(k\pm \epsilon_{j}\pm \epsilon_{j+1})$. In the sequel, we denote $l(A_{1}\cdots A_{M})$ simply by $l$.

Let $k$ be such that $P(k-\epsilon_{1})\not=0$ and $k_{1}\not=0$. If $P(k)=0$ then there is $1\leq j\leq M$ such that $A_{j}=H_{1}$. This is proved by induction on $M$ using the action of $\tilde{S}_{0}$ given by formulae \eqref{eq-action}. In particular, if $k_{1}\not=0$, $P(k)\not=0$ and $P(k-\epsilon_{1})\not=0$, then $\frac{|P(k-\epsilon_{1})|}{|P(k)|}$ is well-defined and is bounded by a number depending on $M$ only.

Now we write
\begin{align}\label{eq-norm2}
\|A_{1}\cdots A_{M}f(E_{0})u\|^2 & = \sum_{{\scriptsize \begin{array}{l} k\ | \ A_{1}\cdots A_{M}e(k)\not=0,\\ A_{1}\cdots A_{M}e(k-\epsilon_{1})\not=0\end{array}}}\ |u_{k}|^2|m^-_{k}|^2\|A_{1}\cdots A_{M}e(k-\epsilon_{1})\|^2 \nonumber\\
 & + \sum_{{\scriptsize \begin{array}{l} k\ | A_{1}\cdots A_{M}e(k)=0,\\ A_{1}\cdots A_{M}e(k-\epsilon_{1})\not=0\end{array}}}\ |u_{k}|^2|m^-_{k}|^2\|A_{1}\cdots A_{M}e(k-\epsilon_{1})\|^2.
\end{align}
Let us work with the first sum. We can rewrite it in the following form:
$$\sum\  |u_{k}|^2|m^-_{k}|^2\frac{\|A_{1}\cdots A_{M}e(k-\epsilon_{1})\|^2}{\|A_{1}\cdots A_{M}e(k)\|^2}\|A_{1}\cdots A_{M}e(k)\|^2.$$ This in turn is equal to:
$$\sum\ |m^-_{k}|^2\frac{|P(k-\epsilon_{1})|^2}{|P(k)|^2}\frac{\|e(k+l-\epsilon_{1}\|^2}{\|e(k+l)\|^2}\times |u_{k}|^2\|A_{1}\cdots A_{M}e(k)\|^2.$$ Therefore we have:
\begin{align*}
 & \sum\ |m^-_{k}|^2\frac{|P(k-\epsilon_{1})|^2}{|P(k)|^2}\frac{\|e(k+l-\epsilon_{1}\|^2}{\|e(k+l)\|^2}\times |u_{k}|^2\|A_{1}\cdots A_{M}e(k)\|^2\\
 & \leq \sup\left\{|m^-_{k}|^2\frac{|P(k-\epsilon_{1})|^2}{|P(k)|^2}\frac{\|e(k+l-\epsilon_{1}\|^2}{\|e(k+l)\|^2}\right\}\ \sum \|A_{1}\cdots A_{M}u_{k}e(k)\|^2\\
 & \leq \sup\left\{|m^-_{k}|^2\frac{|P(k-\epsilon_{1})|^2}{|P(k)|^2}\frac{\|e(k+l-\epsilon_{1}\|^2}{\|e(k+l)\|^2}\right\}\ \|A_{1}\cdots A_{M}u\|^2\\
 & \leq \sup\left\{|m^-_{k}|^2\frac{|P(k-\epsilon_{1})|^2}{|P(k)|^2}\frac{\|e(k+l-\epsilon_{1}\|^2}{\|e(k+l)\|^2}\right\}\ \times \|u\|_{M}^2.
\end{align*}
We must now prove that $\sup\left\{|m^-_{k}|^2\frac{|P(k-\epsilon_{1})|^2}{|P(k)|^2}\frac{\|e(k+l-\epsilon_{1}\|^2}{\|e(k+l)\|^2}\right\}$ is bounded by a number independent of $l$ (but possibly depending on $M$). From previous remarks, it is sufficient to prove that $|m^-_{k}|^2\frac{\|e(k+l-\epsilon_{1})\|^2}{\|e(k+l)\|^2}$ is bounded. Since $|l|\leq 2M$, this is an easy consequence of the explicit expression for $m^-_{k}$.

Let us now investigate the second sum in \eqref{eq-norm2}, assuming it is not empty. As we already mentioned, there is an index $j$ such that $A_{j}=H_{1}$. Then using commutation relations in the enveloping algebra, we have:
$$A_{1}\cdots A_{M}=A'_{1}\cdots A'_{M-1}H_{1}+A''_{1}\cdots A''_{M-1}.$$ Therefore, we have:
\begin{align*}
 & \sum_{k\ | A_{1}\cdots A_{M}e(k)=0,\ A_{1}\cdots A_{M}e(k-\epsilon_{1})\not=0}\ |u_{k}|^2|m^-_{k}|^2\|A_{1}\cdots A_{M}e(k-\epsilon_{1})\|^2\\
 & \leq  \sum\ |u_{k}|^2|m^-_{k}|^2\left(\|A'_{1}\cdots A'_{M-1}H_{1}e(k-\epsilon_{1})\|^2+\|A''_{1}\cdots A''_{M-1}e(k-\epsilon_{1})\|^2\right)\\
 & \leq  \left(\sum\ |u_{k}|^2|m^-_{k}|^2(k_{1}-k_{2}-1)^2\|A'_{1}\cdots A'_{M-1}e(k-\epsilon_{1})\|^2\right) \\
 & \phantom{aaaa} + \|A''_{1}\cdots A''_{M-1}f(E_{0})u\|^2\\
 & \leq \sup\{(k_{1}-k_{2}-1)^2\}\times \|A'_{1}\cdots A'_{M-1}f(E_{0})u\|^2 + \|A''_{1}\cdots A''_{M-1}f(E_{0})u\|^2
\end{align*}
The second sum is by induction smaller than $\|f(E_{0})\|^2_{_{M-1}}\times \|u\|^2_{_{M-1}}$. For the first sum, the induction shows that $\|A'_{1}\cdots A'_{M-1}f(E_{0})u\|^2\leq \|f(E_{0})\|^2_{_{M-1}}\times \|u\|^2_{_{M-1}}$. Thus, it suffices to prove that $(k_{1}-k_{2}-1)^2$ is bounded. Since $A_{1}\cdots A_{M}e(k)=0$, there is an integer $j$ such that 
$$A_{j}=H_{1}, \ H_{1}A_{j+1}\cdots A_{M}e(k)=0 \mbox{ and } A_{j+1}\cdots A_{M}e(k)\not=0.$$ But then $A_{j+1}\cdots A_{M}e(k)=C\times e(k+l')$ for some non zero constant $C$. As above $|l'|\leq 2M$. Thus $H_{1}e(k+l')=0$, which means that $(k+l')_{1}=(k+l')_{2}$ or also $k_{1}-k_{2}-1=l'_{1}-l'_{2}-1$. And we have $|l'_{1}-l'_{2}-1|\leq |l'|+1\leq 2M+1$, proving thus that $\sup\{(k_{1}-k_{2}-1)^2\}$ is bounded by a number depending on $M$ only.

Altogether, we have proved that:
$$\|A_{1}\cdots A_{M}f(E_{0})u\|^2 \leq C(M)\times \|u\|^2_{_{M}},$$ for some constant $C(M)$ depending on $M$ only (Note here that $\|u\|_{_{M-1}}\leq \|u\|_{_{M}}$). As a consequence, we get $$\|f(E_{0})u\|_{_{M}}\leq \sqrt{C(M)}\|u\|_{_{M}},$$ proving that $f(E_{0})$ is bounded for the norm $\|\cdot \|_{_{M}}$.
\end{proof}

\begin{corollary}\label{coro-intNa}
Let $a\in \C\setminus\N$. Then the representation $\cal H _{a}$ of $\germ g $ integrates into a continuous representation of $G$ on the Hilbert space $\cal H $.
\end{corollary}

\begin{remark}
Let $a\in \N$. Define a representation $\cal H _{a}$ as above by restricting the index set of $k$ to those $k\in \N^n$ such that $|k|\leq a$. Then $\cal H _{a}$ is indeed a representation and is finite dimensional. Therefore it also integrates into a continuous representation of $G$ on some Hilbert space.
\end{remark}

\subsection{Unitarisability}

We now know a whole family of continuous representation of $G$. We should ask then which of these are unitary. If the representation $\cal H _{a}$ is unitary then the infinitesimal action given by the Lie basis $\{iH_{j}(a),\ X_{j}(a), \ Y_{j}(a)\}$ should be given by skew-symmetric operators. In other word, we should have $(iH_{j}(a))^*=-iH_{j}(a)$, $X_{j}(a)^*=-X_{j}(a)$ and $Y_{j}(a)^*=-Y_{j}(a)$. Using the expression of the $H_{j}(a), E_{j}(a), F_{j}(a)$ in term of this basis, it is equivalent to have $H_{j}(a)^*=H_{j}(a)$, $E_{j}(a)^*=F_{j}(a)$ for $j>0$ and $E_{0}(a)^*=-F_{0}(a)$. As $H_{j}(a)$ is a diagonal operator, it is selfadjoint if and only if its eigenvalues are real. This imposes $a\in \R$. Remember that the representation that we started with is unitary. So it only remains to prove $f(E_{j})^*=f(F_{j})$ for $j>0$ and $f(E_{0})^*=-f(F_{0})$. The first condition is trivial since $f(E_{j})=f(F_{j})=0$. To check the second condition, we need to compare $\langle f(F_{0})e(k),e(l)\rangle$ and $-\langle e(k),f(E_{0})e(l)\rangle$. They are equal for all $k$ and $l$ if and only if $a\in \R_{<0}$.

\begin{proposition}\label{prop-unitNa}
The continuous representation $\cal H _{a}$ of $G$ is unitary if and only if $a\in \R_{<0}$.
\end{proposition}

\begin{remark}
When $a\in\N$, the representation $\cal H _{a}$ constructed in the remark \ref{rema-fdd1} is not unitary (unless $a=0$) since it has finite dimension greater than $1$ and $G$ is not a compact group.
\end{remark}

\subsection{$SU(p,q)$ case}

Let $1\leq p\leq n$. Set $q=n+1-p$. Let $G_{p,q}$ denote the universal cover of $SU(p,q)$. The complex Lie algebra $\germ g =\germ sl (n+1,\C)$ is the complexification of the Lie algebra of $G_{p,q}$. Moreover, $G_{p,q}$ contains a compact subgroup $K_{p,q}$ isomorphic to $SU(p)\times SU(q)$, whose complexified Lie algebra is isomorphic to the semisimple part of $\germ l _{p-1}$. Let us now give the classification of all simple infinite dimensional degree $1$ modules coming from a continuous representation of $G_{p,q}$ on some Hilbert space.

\begin{theorem}\label{theo:typeA}
Let $V$ be a simple infinite dimensional weight $\germ g $-module of degree 1. Then $V$ integrates into a continuous representation of $G_{p,q}$ on a Hilbert space if and only if
\begin{enumerate}
\item Either $V$ or its contragredient is isomorphic to $N(a,0,\ldots,0)$ (for $a\in \C\setminus\N$) or to $N(-1,m,0,\ldots,0)$ (for $m\in \N$), in case $p=1$.
\item Either $V$ or its contragredient is isomorphic to $N(-1,\ldots,-1,a)$ (for $a\in \C\setminus\Z_{<0}$) or to $N(-1,\ldots,-1,-1-m,0)$ (for $m\in \N$), in case $p=n$.
\item Either $V$ or its contragredient is isomorphic to $N(\underbrace{-1,\ldots,-1}_{p},m,0,\ldots,0)$ (for $m\in\N$) or to $N(\underbrace{-1,\ldots,-1}_{p-1},-1-m,0,\ldots 0)$ (for $m\in \N$), in case $1<p<n$.
\end{enumerate}
Moreover, the corresponding representation of $G_{p,q}$ is unitary if and only if
\begin{enumerate}
\item Either $V$ or its contragredient is isomorphic to $N(a,0,\ldots,0)$ (for $a\in \R_{<0}$) or to $N(-1,m,0,\ldots,0)$ (for $m\in \N$), in case $p=1$.
\item Either $V$ or its contragredient is isomorphic to $N(-1,\ldots,-1,a)$ (for $a\in\R_{>0}$) or to $N(-1,\ldots,-1,-1-m,0)$ (for $m\in \N$), in case $p=n$.
\item Either $V$ or its contragredient is isomorphic to $N(\underbrace{-1,\ldots,-1}_{p},m,0,\ldots,0)$ (for $m\in\N$) or to $N(\underbrace{-1,\ldots,-1}_{p-1},-1-m,0,\ldots 0)$ (for $m\in \N$), in case $1<p<n$.
\end{enumerate}
\end{theorem}
\begin{proof} First, remark that given any continuous representation $\pi$ of $G_{p,q}$ on a Hilbert space, its restriction to $K_{p,q}$ splits into a direct sum of finite dimensional representations (possibly with infinite multiplicities). Therefore, the corresponding $\germ g $-module should also split into a direct sum of finite dimensional $\germ l _p$-modules. In other word, a necessary condition is that the underlying $\germ g $-module is a $(\germ g ,\germ l _p)$-module of finite type. Therefore, by proposition \ref{prop-glmod}, $V$ or its contragredient should be isomorphic to the asserted modules.
\begin{enumerate}
\item In case $p=1$, we already know from corollary \ref{coro-intNa}, that $N(a,0,\ldots,0)$ does integrate into a continuous representation of $G_1$. Moreover from proposition \ref{prop-unitNa}, we know that this representation is unitary exactly when $a\in \R_{<0}$. Now, the module $N(-1,m,0,\ldots,0)$ is a highest weight module with highest weight $\lambda_m=(-1-m,m,0,\ldots,0)\in \Z^n$. This is clearly an analytically integral weight, and dominant with respect to the positive roots of $\germ l _0$. It is then straightforward to check that it is the underlying $\germ g $-module of the holomorphic discrete series of $SU(1,n)$ corresponding to the parameter $\lambda_m$.
\item The case $p=n$ is of course identical to the previous one up to a relabeling of the simple roots.
\item The intermediate case $1<p<n$ is easy, since the possible underlying $\germ g $-modules all correspond to holomorphic discrete series, their parameter being $(0,\ldots,0,-1-m,m,0,\ldots,)$ or $(0,\ldots ,0,m,-1-m,0,\ldots,0)$.
\end{enumerate}
\end{proof}

\subsection{$SL(n,\R)$ case}

Assume that $n\geq 3$. Let $G_{n}$ denote the universal cover of $SL(n,\R)$. Its complexified Lie algebra is also $\germ g =\germ sl (n,\C)$. The compact Lie group $K_{n}=SO(n)$ is a subgroup of $G_{n}$.

\begin{theorem}
Let $V$ be a simple weight $\germ g $-module. Then $V$ can be integrated into a continuous representation of $G_{n}$ on a Hilbert space if and only if $V$ is finite dimensional.
\end{theorem}
\begin{proof}
Assume $V$ can be integrated into a continuous representation of $G_{n}$ in a Hilbert space. The complexified Lie algebra of $K_{n}$ contains the vectors $E_{j}+F_{j}$. By lemma \ref{Fe-BBL}, these vectors should act locally finitely on $V$. Let $V_{\lambda}$ be a weight space of $V$. Denote by $\alpha_{j}$ the weight of $E_{j}$. Then $E_{j}+F_{j}: V_{\lambda}\rightarrow V_{\lambda+\alpha_{j}}\oplus V_{\lambda-\alpha_{j}}$. Therefore $E_{j}+F_{j}$ is locally finite if and only if both $E_{j}$ and $F_{j}$ are. Then the module is finite dimensional as asserted. The converse is obvious.
\end{proof}

\section{Type $C$ case}

Let $n$ be a positive integer. Let $p$ and $q$ be positive integers such that $p+q=n$. In this section, we consider the groups $Sp(n,\R)$ and $Sp(p,q)$ and their universal cover $G_n$ and $G_{p,q}$. They contain the compact subgroup $K_n$ and $K_{p,q}$ isomorphic to $SU(n)$ and $SP(p)\times Sp(q)$ respectively (see for instance \cite{Kna02}). Denote by $\germ g $ the Lie algebra $\germ sp (n,\C)$.

\begin{theorem}
Let $V$ be a simple infinite dimensional weight $\germ g $-module of degree 1. Then 
\begin{enumerate}
\item $V$ cannot integrate into a continuous representation of $G_{p,q}$ on a Hilbert space.
\item $V$ integrates into a continuous representation of $G_n$ on a Hilbert space if and only if $V$ or its contragredient is isomorphic to $M(-1,\ldots ,-1)$ or $M(-1,\ldots ,-1,-2)$. In this case, the corresponding representation of $G_n$ is simple and unitary, and isomorphic to the even or odd part of the metaplectic representation or its contragredient.
\end{enumerate}
\end{theorem}
\begin{proof} The proof is analogous to the proof of theorem \ref{theo:typeA}. In the case of $G_{p,q}$, the complexified Lie algebra of $K_{p,q}$ is the Lie algebra $\germ l _p$ whose roots with respect to the standard Cartan subalgebra of $\germ g $ are (see \cite[Appendix $C$]{Kna02})
\begin{align*}
 & \{\pm 2\epsilon_l,\ \pm(\epsilon_j\pm\epsilon_k)\ : \ 1\leq l\leq p,\ 1\leq j\not=k\leq p\}\\
\sqcup & \{\pm 2\epsilon_l,\ \pm(\epsilon_j\pm\epsilon_k)\ : \ p+1\leq l\leq p+q,\ p+1\leq j\not=k\leq p+q\}.
 \end{align*}
The module $V$ should be a $(\germ g ,\germ l _p)$-module of finite type. Using an analogue of proposition \ref{prop-glmod}, we see that $V$ or its contragredient should be a highest weight module. The simple infinite dimensional highest weight module of degree 1 have been classified by Benkart, Britten, Lemire \cite[Proposition 3.6]{BBL97}. They are only two:  their highest weights are $-\frac{1}{2}\omega_{n}$ and $\omega_{n-1}-\frac{3}{2}\omega_{n}$ respectively. It is then easy to check that the possible modules are not $(\germ g ,\germ l _p)$-module of finite type.

The case of $G_n$ is analogous. The complexified Lie algebra of $K_n$ is the Lie algebra $\germ l $ whose roots are $$\{\pm(\epsilon_j\pm\epsilon_k)\ : \ 1\leq j\not=k\leq n-1\}.$$ Once again, $V$ or its contragredient should be isomorphic to a highest weight module. It is then well-known that these two highest weight modules correspond to the even and odd part of the metaplectic representation (see e.g. \cite{Li00}).
\end{proof}

\vspace{0.5cm}
\noindent{\bf Acknowledgment}: We are grateful to Professor Bent {\O}rsted for some very interesting and stimulating discussions of the subjects. The author also gratefully acknowledges support from Aarhus University, and the Max Planck Institut for Mathematics in Bonn.



\end{document}